\documentclass[12pt, reqno]{amsart}
\usepackage{amsmath, amsthm, amscd, amsfonts, amssymb, graphicx, color}
\usepackage[bookmarksnumbered, colorlinks, plainpages]{hyperref}
\hypersetup{colorlinks=true,linkcolor=red, anchorcolor=green, citecolor=cyan, urlcolor=red, filecolor=magenta, pdftoolbar=true}

\newtheorem{theorem}{Theorem}[section]

\newtheorem{proposition}[theorem]{Proposition}

\theoremstyle{definition}

\theoremstyle{remark}

\numberwithin{equation}{section}

\begin{document}

\setcounter{page}{1}

\title[Van der Corput lemmas for Mittag-Leffler functions. II.]{Van der Corput lemmas for Mittag-Leffler functions. II. $\alpha$--directions}

\author[M. Ruzhansky, B. T. Torebek]{Michael Ruzhansky, Berikbol T. Torebek}

\address{\textcolor[rgb]{0.00,0.00,0.84}{Michael Ruzhansky \newline Department of Mathematics: Analysis,
Logic and Discrete Mathematics \newline Ghent University, Krijgslaan 281, Ghent, Belgium \newline
 and \newline School of Mathematical Sciences \newline Queen Mary University of London, United Kingdom}}
\email{\textcolor[rgb]{0.00,0.00,0.84}{michael.ruzhansky@ugent.be}}
\address{\textcolor[rgb]{0.00,0.00,0.84}{Berikbol T. Torebek \newline Department of Mathematics: Analysis,
Logic and Discrete Mathematics \newline Ghent University, Krijgslaan 281, Ghent, Belgium \newline and \newline Al--Farabi Kazakh National University \newline Al--Farabi ave. 71, 050040, Almaty, Kazakhstan \newline and \newline Institute of
Mathematics and Mathematical Modeling \newline 125 Pushkin str.,
050010 Almaty, Kazakhstan}}
\email{\textcolor[rgb]{0.00,0.00,0.84}{berikbol.torebek@ugent.be}}

\dedicatory{Dedicated to  Johannes Gaultherus van der Corput on the occasion of 100 years of his lemma}
\thanks{The authors were supported in parts by the FWO Odysseus 1 grant G.0H94.18N: Analysis and Partial Differential Equations. The first author was supported by EPSRC grant EP/R003025/1 and by the Leverhulme Grant RPG-2017-151.}
\date{\today}



\begin{abstract} The paper is devoted to study analogues of the van der Corput lemmas involving Mittag-Leffler functions. The generalisation is that we replace the exponential function with the Mittag-Leffler-type function, to study oscillatory integrals appearing in the analysis of time-fractional partial differential equations. More specifically, we study integral of the form
$I_{\alpha,\beta}(\lambda)=\int_\mathbb{R}E_{\alpha,\beta}\left(i^\alpha\lambda \phi(x)\right)\psi(x)dx,$
for the range $0<\alpha\leq 2,\,\beta>0$. This extends the variety of estimates obtained in the first part, where integrals with functions $E_{\alpha,\beta}\left(i \lambda \phi(x)\right)$ have been studied.
Several generalisations of the van der Corput lemmas are proved. As an application of the above results, the generalised Riemann-Lebesgue lemma, the Cauchy problem for the time-fractional Klein-Gordon and time-fractional Schr\"{o}dinger  equations are considered.
\end{abstract}
\maketitle
\tableofcontents
\section{Introduction}

In this paper we continue the study of oscillatory-type integrals involving Mittag-Leffler functions $E_{\alpha,\beta}$ initiated in \cite{RuzTor}. In the case of $\alpha=\beta=1$, we have $E_{1,1}(z)=e^z$, thus reducing the integral to the classical question of decay of oscillatory integrals.

Indeed, the estimate obtained by the Dutch mathematician Johannes Gaultherus van der Corput \cite{vdC21} and named in his honour, following Stein \cite{St93}, can be stated as follows:
\begin{itemize}
  \item \textbf{van der Corput lemma.} Suppose $\phi$ is a real-valued and smooth function in $[a,b].$ If $\psi$ is a smooth function and $|\phi^{(k)}(x)|\geq 1,\,k\geq 1,$ for all $x\in(a,b),$ then
  \begin{equation}\label{vdC}
  \left|\int\limits^b_a e^{i\lambda\phi(x)}\psi(x)dx\right|\leq C\lambda^{-1/k},\,\,\,\lambda\rightarrow\infty,
  \end{equation} for $k=1$ and $\phi'$ is monotonic, or $k\geq 2.$ Here $C$ does not depend on $\lambda.$
\end{itemize}

Various generalisations of the van der Corput lemmas have been investigated over the years \cite{Gr05, SW70, St93, PS92, PS94, Rog05, Par08, Xi17}. Multidimensional analogues of the van der Corput lemmas were studied in \cite{Bourg, CCW99, Tao05, Green, PSS01, KR07}, while in \cite{Ruz12} the multi-dimensional van der Corput lemma was obtained with constants independent of the phase and amplitude.

The main goal of the present paper is to study van der Corput lemmas for the oscillatory integral defined by
\begin{equation}\label{1}I_{\alpha,\beta}(\lambda)=\int\limits_\mathbb{R}E_{\alpha,\beta}\left(i^\alpha\lambda \phi(x)\right)\psi(x)dx,\end{equation}
where $0<\alpha< 2,\,\beta>0,$ $\phi$ is a phase and $\psi$ is an amplitude, and $\lambda$ is a positive real number that can vary. Here $E_{\alpha,\beta}\left(z\right)$ is the Mittag-Leffler function defined as (see e.g. \cite{Kilbas, GorKMR})
\begin{equation*}E_{\alpha,\beta}\left(z\right)=\sum\limits_{k=0}^\infty\frac{z^k}{\Gamma(\alpha k+\beta)},\,\,\,\alpha>0, \,\,\,\beta\in \mathbb{R},\end{equation*} with the property that
\begin{equation}\label{EQ:e11}
E_{1,1}\left(z\right)=e^z.
\end{equation}
Since the function $E_{\alpha,\beta}\left(i^\alpha z\right),\,z\in\mathbb{R}$, has a set of real zeros \cite{GorKMR},  the integral \eqref{1} is oscillating.

Here we can point out already one extension of \eqref{vdC} in view of \eqref{EQ:e11}, namely, an extension (in Theorem \ref{th1.4}) to the range $0<\alpha<2$ in the form
  \begin{equation}\label{vdCa}
  \left|\int\limits^b_a E_{\alpha,\alpha}\left({i^\alpha\lambda\phi(x)}\right)\psi(x)dx\right|\leq C\lambda^{-1/k},\,\,\,\lambda\rightarrow\infty,
  \end{equation} for $k=1$ and $\phi'$ is monotonic, or $k\geq 2.$

This present paper is a continuation of \cite{RuzTor}, where a variety of van der Corput type lemmas were obtained for the integral defined by
\begin{equation}\label{1+} \widetilde{I}_{\alpha,\beta}(\lambda)=\int\limits_\mathbb{R}E_{\alpha,\beta}\left(i\lambda \phi(x)\right)\psi(x)dx,\end{equation}
where $0<\alpha< 1,\,\beta>0.$

As we see above, the integral \eqref{1+} is different from the integral \eqref{1}, since in \eqref{1+} there is a purely imaginary number $i$ before the phase function, and in \eqref{1} the fractional power of the imaginary number, i.e. $i^\alpha$. In addition, the asymptotic behavior of the Mittag-Leffler function in these cases is also different, yielding different decay rates.

Such integrals as in \eqref{1} arise in the study of decay estimates of solutions of the time-fractional Schr\"{o}dinger and the time-fractional wave equations (for example see \cite{Dong, Grande, Naber, Su}). In Section \ref{SEC:apps} we will give several immediate applications of the obtained estimates to time-fractional Klein-Gordon and Schr\"odinger equations.

As in the case of \eqref{1+} studied in \cite{RuzTor}, we find that the decay rates of \eqref{1} as $\lambda\to\infty$ depend not only on the assumptions on the phase but also on the ranges of parameters $\alpha$ and $\beta$. We also obtain more results in the case of bounded intervals.
For the convenience of the reader, let us briefly summarise the results of this paper, distinguishing between different sets of assumptions:\\

\noindent
\underline{\textbf{van der Corput lemmas on $\mathbb{R}$}}: consider $I_{\alpha,\beta}$ defined by \eqref{1}.
\begin{itemize}
  \item Let $\phi:\mathbb{R}\rightarrow \mathbb{R}$ be a measurable function and let $\psi\in L^1(\mathbb{R}).$ Suppose that $0<\alpha\leq 2,\,\beta>1,$ and $m=\operatorname{ess\,inf}\limits_{x\in\mathbb{R}}|\phi(x)|> 0,$ then
\begin{description}
  \item[(i)] for $0<\alpha<2$ and $\beta\geq \alpha+1$ we have
\begin{equation*}|I_{\alpha,\beta}(\lambda)|\leq \frac{M_1}{1+\lambda m}\|\psi\|_{L^1(\mathbb{R})},\,\lambda\geq 1,
\end{equation*} where $M_1$ does not depend on $\phi,$ $\psi$ and $\lambda;$
  \item[(ii)] for $0<\alpha<2$ and $1<\beta< \alpha+1$ we have
\begin{equation*}|I_{\alpha,\beta}(\lambda)|\leq \frac{M_2}{(1+\lambda m)^\frac{\beta-1}{\alpha}}\|\psi\|_{L^1(\mathbb{R})},\,\lambda\geq 1,
\end{equation*} where $M_2$ does not depend on $\phi,$ $\psi$ and $\lambda;$
\item[(iii)] for $\alpha=2$ and $\beta>1$ we have
\begin{equation*}|I_{2,\beta}(\lambda)|\leq \frac{M_3}{(1+\lambda m)^\frac{\beta-1}{2}}\|\psi\|_{L^1(\mathbb{R})},\,\lambda\geq 1,\end{equation*} where $M_3$ does not depend on $\phi,$ $\psi$ and $\lambda.$
\end{description}
\item Let $\phi:\mathbb{R}\rightarrow \mathbb{R}$ be an invertible and differentiable function, and let $\psi\in L^1(\mathbb{R}).$ Suppose that $0<\alpha\leq 2,\,\beta=1,$ and $m=\inf\limits_{x\in\mathbb{R}}|\phi'(x)|> 0,$ then
\begin{equation*}|I_{\alpha,1}(\lambda)|\leq \frac{M}{\lambda m}\|\psi\|_{L^1(\mathbb{R})},\,\lambda\geq 1,
\end{equation*} where $M$ does not depend on $\phi,$ $\psi$ and $\lambda.$
\end{itemize}
\underline{\textbf{van der Corput lemmas on $I=[a,b]\subset\mathbb{R}$}}: consider $$I_{\alpha,\beta}(\lambda)=\int\limits_I E_{\alpha,\beta}\left(i^\alpha\lambda \phi(x)\right)\psi(x)dx,$$
where $I=[a,b]\subset \mathbb{R}$ with $-\infty< a<b<+\infty$.
\begin{itemize}
  \item Let  $0<\alpha<2,\,\beta>1,$ $\phi$ be a real-valued function such that $\phi\in C^k(I),\,k\geq 1,$ and let $\psi\in C^1(I).$ If
  $|\phi^{(k)}(x)|\geq 1$ for all $x\in I,$ then
\begin{description}
  \item[(i)] for $0<\alpha<2$ and $\beta\geq \alpha+1$ we have
\begin{equation*}
\left|I_{\alpha,\beta}\left(\lambda \right)\right|\leq M_k\left[|\psi(b)|+\int\limits_I|\psi'(x)|dx\right]\lambda^{-\frac{1}{k}}\log^{\frac{1}{k}}(1+\lambda),\,\,\,\lambda\geq 1,
\end{equation*} where $M_k$ does not depend on $\lambda;$
  \item[(ii)] for $0<\alpha<2$ and $1<\beta< \alpha+1$ we have
\begin{equation*}
\left|I_{\alpha,\beta}\left(\lambda \right)\right|\leq M_k\left[|\psi(b)|+\int\limits_I|\psi'(x)|dx\right]\lambda^{-\frac{1}{k}}(1+\lambda)^{\frac{\alpha+1-\beta}{\alpha k}},\,\,\,\lambda\geq 1,\end{equation*} where $M_k$ does not depend on $\lambda.$
\end{description}
\item Let $-\infty< a<b<+\infty$ and $I=[a,b]\subset \mathbb{R}.$ Let  $0<\alpha<2$ and let $\phi$ be a real-valued function such that $\phi\in C^k(I),\,k\geq 1.$ Let $\psi\in C^1(I)$ and $|\phi^{(k)}(x)|\geq 1$ for all $x\in I.$ Then
\begin{equation*}
\left|I_{\alpha,\alpha}\left(\lambda\right)\right|\leq M_k\left[|\psi(b)|+\int\limits_I|\psi'(x)|dx\right]\lambda^{-1/k},\,\,\,\lambda\geq 1,
\end{equation*} for $k=1$ and $\phi'$ is monotonic, or $k\geq 2$. Here $M_k$ does not depend on $\lambda.$
\item Let $-\infty< a<b<+\infty$ and $I=[a,b]\subset \mathbb{R}.$ Let  $0<\alpha<2$ and let $\phi$ be a real-valued function such that $\phi\in C^2(I).$ Let $|\phi'(x)|\geq 1$ for all $x\in I.$ Then
\begin{equation*}
\left|I_{\alpha,\alpha}\left(\lambda\right)\right|\leq M\left[|\psi(b)|+\int\limits_I|\psi'(x)|dx\right]\lambda^{-1},\,\,\,\lambda\geq 1,
\end{equation*} where $M$ does not depend on $\lambda.$
\end{itemize}
We will often make use of the following estimate.
\begin{proposition}[\cite{Pod99}] If $0<\alpha<2,$ $\beta$ is an arbitrary real number, $\mu$ is such that $\pi\alpha/2<\mu<\min\{\pi, \pi\alpha\}$, then there is $C_1,C_2>0,$ such that we have
\begin{equation}\label{MLAsym} \left|E_{\alpha, \beta}(z)\right|\leq C_1(1+|z|)^{(1-\beta)/\alpha}\exp({\rm Re}(z^{1/\alpha}))+ \frac{C_2}{1+|z|},\, z\in \mathbb{C},\, |\arg(z)|\leq\mu.\end{equation}
\end{proposition}

We are interested in particular in the behavior of $I_{\alpha,\beta}(\lambda)$ when $\lambda$ is large, as for small $\lambda$ the integral is just bounded.

\section{Van der Corput lemma in $\mathbb{R}$}

In this section we consider $I_{\alpha,\beta}$ defined by \eqref{1}, that is,
$$I_{\alpha,\beta}(\lambda)=\int\limits_\mathbb{R}E_{\alpha,\beta}\left(i^\alpha\lambda \phi(x)\right)\psi(x)dx.$$
As for small $\lambda$ the integral \eqref{1} is just bounded, we consider the case $\lambda\geq 1.$
\begin{theorem}\label{th1} Let $\phi:\mathbb{R}\rightarrow \mathbb{R}$ be a measurable function and let $\psi\in L^1(\mathbb{R}).$ Suppose that $0<\alpha\leq 2,\,\beta>1,$ and $m=\operatorname{ess\,inf}\limits_{x\in\mathbb{R}}|\phi(x)|> 0,$ then
\begin{description}
  \item[(i)] for $0<\alpha<2$ and $\beta\geq \alpha+1$ we have
\begin{equation}\label{2}|I_{\alpha,\beta}(\lambda)|\leq \frac{M_1}{1+\lambda m}\|\psi\|_{L^1(\mathbb{R})},\,\lambda\geq 1,
\end{equation} where $M_1$ does not depend on $\phi,$ $\psi$ and $\lambda;$
  \item[(ii)] for $0<\alpha<2$ and $1<\beta< \alpha+1$ we have
\begin{equation}\label{2*}|I_{\alpha,\beta}(\lambda)|\leq \frac{M_2}{(1+\lambda m)^\frac{\beta-1}{\alpha}}\|\psi\|_{L^1(\mathbb{R})},\,\lambda\geq 1,
\end{equation} where $M_2$ does not depend on $\phi,$ $\psi$ and $\lambda;$
\item[(iii)] for $\alpha=2$ and $\beta>1$ we have
\begin{equation}\label{2*}|I_{2,\beta}(\lambda)|\leq \frac{M_3}{(1+\lambda m)^\frac{\beta-1}{2}}\|\psi\|_{L^1(\mathbb{R})},\,\lambda\geq 1,\end{equation} where $M_3$ does not depend on $\phi,$ $\psi$ and $\lambda.$
\end{description}
\end{theorem}
\begin{proof} Let $\phi:\mathbb{R}\rightarrow \mathbb{R}$ be a measurable function and $\psi\in L^1(\mathbb{R}).$ As $|\arg(i^\alpha\lambda\phi(x))|=\frac{\pi\alpha}{2}$ and ${\rm Re}(i\lambda^{1/\alpha}(\phi(x))^{1/\alpha})=0,$ then using estimate \eqref{MLAsym} we have that
\begin{align*}
|I_{\alpha,\beta}(\lambda)|& \leq \int\limits_\mathbb{R}\left|E_{\alpha,\beta}\left(i^\alpha\lambda \phi(x)\right)\right|\left|\psi(x)\right|dx\\& \leq C_1\int\limits_\mathbb{R}(1+\lambda|\phi(x)|)^{(1-\beta)/\alpha}\left|\psi(x)\right|dx+ C_2\int\limits_\mathbb{R}\frac{\left|\psi(x)\right|}{1+\lambda|\phi(x)|}dx.
\end{align*}

As $\phi$ and $\psi$ do not depend on $\lambda,$ and $m=\operatorname{ess\,inf}\limits_{x\in\mathbb{R}}|\phi(x)|> 0,$
then for $\beta\geq \alpha+1$ we have
\begin{align*}
|I_{\alpha,\beta}(\lambda)|&\leq \int\limits_\mathbb{R}\left|E_{\alpha,\beta}\left(i^\alpha\lambda \phi(x)\right)\right|\left|\psi(x)\right|dx\\& \leq \max\{C_1,C_2\}\int\limits_\mathbb{R}\frac{\left|\psi(x)\right|}{1+\lambda|\phi(x)|}dx\\& \leq \frac{M_1}{1+\lambda m}\|\psi\|_{L^1(\mathbb{R})}.
\end{align*}
In the case $1<\beta< \alpha+1$ we have that
\begin{align*}
|I_{\alpha,\beta}(\lambda)|&\leq \int\limits_\mathbb{R}\left|E_{\alpha,\beta}\left(i^\alpha\lambda \phi(x)\right)\right|\left|\psi(x)\right|dx\\& \leq \max\{C_1,C_2\}\int\limits_\mathbb{R}\frac{\left|\psi(x)\right|}{(1+\lambda|\phi(x)|)^\frac{\beta-1}{\alpha}}dx\\& \leq \frac{M_2}{(1+\lambda m)^\frac{\beta-1}{\alpha}}\|\psi\|_{L^1(\mathbb{R})}.
\end{align*} The cases (i) and (ii) are proved.

Now we will prove the case (iii). Applying the asymptotic estimate (see \cite[page 43]{Kilbas})
\begin{align*}
E_{2,\beta}(z)&=\frac{1}{2}z^{(1-\beta)/2}\left(e^{\sqrt{z}}+e^{-\sqrt{z}-\pi i(1-\beta)sign(\arg z)}\right) \\& -\sum\limits_{k=1}^N\frac{z^{-k}}{\Gamma(\beta-2k)}+O\left(\frac{1}{z^{N+1}}\right),\,|z|\rightarrow\infty,\,|\arg(z)|\leq \pi,
\end{align*}
we have
\begin{align*}
|I_{2,\beta}(\lambda)|&\leq \int\limits_\mathbb{R}\left|E_{2,\beta}\left(-\lambda \phi(x)\right)\right|\left|\psi(x)\right|dx \\&\leq M_3\lambda^{(1-\beta)/2}\int\limits_\mathbb{R}|\phi(x)|^{(1-\beta)/2}\left|\psi(x)\right|dx
\\&{\leq} M_3m^{(1-\beta)/2}\lambda^{(1-\beta)/2}\int\limits_\mathbb{R}\left|\psi(x)\right|dx\\&\leq \frac{M_3\|\psi\|_{L^1(\mathbb{R})}}{(1+m\lambda)^{(\beta-1)/2}}.
\end{align*} Here $M_3$ is a constant that does not depend on $\lambda.$
The proof is complete.
\end{proof}

It is easy to see from estimate \eqref{MLAsym} that, for $\beta = 1,$ the function $E_{\alpha,1}(\cdot)$ does not decrease, and it will be only bounded function. In this case, the method of proving the Theorem \ref{th1} does not suitable. Below we give an estimate for the integral \eqref{1} in the case $\beta=1.$

\begin{theorem}\label{th2} Let $\phi:\mathbb{R}\rightarrow \mathbb{R}$ be an invertible and differentiable function, and let $\psi\in L^1(\mathbb{R}).$ Suppose that $0<\alpha\leq 2,\,\beta=1,$ and $m=\inf\limits_{x\in\mathbb{R}}|\phi'(x)|> 0,$ then
\begin{equation}|I_{\alpha,1}(\lambda)|\leq \frac{M}{\lambda m}\|\psi\|_{L^1(\mathbb{R})},\,\lambda\geq 1,
\end{equation} where $M$ does not depend on $\phi,$ $\psi$ and $\lambda.$
\end{theorem}
\begin{proof} First, we prove the case $0<\alpha< 2$ and $\beta=1.$ Let $\phi:\mathbb{R}\rightarrow \mathbb{R}$ be an invertible and differentiable function, and let $\psi\in L^1(\mathbb{R}).$ Then replacing $\lambda\phi(x)$ by $y$ we get
\begin{align*}I_{\alpha,1}(\lambda)=\frac{1}{\lambda}\int\limits_\mathbb{R}E_{\alpha,1}\left(i^\alpha y\right)\psi((\phi^{-1}(y/\lambda)))\frac{dy}{\phi'(\phi^{-1}(y/\lambda))}.\end{align*}
Since $x\in\mathbb{R},$ then $\phi^{-1}(y/\lambda)\in\mathbb{R}.$ As $|\arg(i^\alpha\lambda\phi(x))|=\frac{\pi\alpha}{2}$ and ${\rm Re}(i\lambda^{1/\alpha}(\phi(x))^{1/\alpha})=0,$ then using \eqref{MLAsym} and $m=\inf\limits_{x\in\mathbb{R}}|\phi'(x)|> 0$ we have
\begin{align*}\left|I_{\alpha,1}(\lambda)\right|&\leq\frac{1}{m\lambda}\int\limits_\mathbb{R}\left|E_{\alpha,1}\left(i^\alpha y\right)\right|\left|\psi((\phi^{-1}(y/\lambda)))\right|dy\\& \leq \frac{M}{m\lambda}\int\limits_\mathbb{R}\left|\psi((\phi^{-1}(y/\lambda)))\right|dy =\frac{M}{m\lambda}\|\psi\|_{L^1(\mathbb{R})}.\end{align*}

If $\alpha=2$ and $\beta=1,$ we have $E_{2,1}(i^2\lambda\phi(x))=\cos\sqrt{\lambda\phi(x)}.$ As $|\cos\sqrt{\lambda\phi(x)}|\leq 1,$ then repeating above calculations we have
\begin{align*}\left|I_{2,1}(\lambda)\right|&\leq \frac{1}{m\lambda}\|\psi\|_{L^1(\mathbb{R})}.\end{align*}
The proof is complete.
\end{proof}

\section{Van der Corput lemma in finite interval}
In this section we consider integral \eqref{1} in the finite interval $I=[a,b]\subset \mathbb{R},\,-\infty< a<b<+\infty,$ i.e.
\begin{equation}\label{1*}I_{\alpha,\beta}(\lambda)=\int\limits_IE_{\alpha,\beta}\left(i^\alpha\lambda \phi(x)\right)\psi(x)dx.\end{equation}
Since $I_{\alpha,\beta}(\lambda)$ is bounded for small $\lambda,$ further we can assume that $\lambda\geq 1.$
\begin{theorem}\label{th2.1} Let  $0<\alpha<2,\,\beta>1,$ $\phi$ be a real-valued function such that $\phi\in C^k(I),\,k\geq 1.$ If $|\phi^{(k)}(x)|\geq 1$ for all $x\in I,$ then
\begin{description}
  \item[(i)] for $0<\alpha<2$ and $\beta\geq \alpha+1$ we have
\begin{equation}\label{2-1}
\left|\int\limits_IE_{\alpha,\beta}\left(i^\alpha\lambda \phi(x)\right)dx\right|\leq M_k\lambda^{-\frac{1}{k}}\log^{\frac{1}{k}}(1+\lambda),\,\,\,\lambda\geq 1,
\end{equation} where $M_k$ does not depend on $\lambda;$
  \item[(ii)] for $0<\alpha<2$ and $1<\beta< \alpha+1$ we have
\begin{equation}\label{2-1*}
\left|\int\limits_IE_{\alpha,\beta}\left(i^\alpha\lambda \phi(x)\right)dx\right|\leq M_k\lambda^{-\frac{1}{k}}(1+\lambda)^{\frac{\alpha+1-\beta}{\alpha k}},\,\,\,\lambda\geq 1,
\end{equation} where $M_k$ does not depend on $\lambda.$
\end{description}
\end{theorem}
\begin{proof} {\bf Proof of (i).}
First we will prove the case $k=1.$
Since $|\phi'(x)|\geq 1$ for all $x\in I,$ then $|\phi|$ is monotonic, and it can have one zero at $c\in[a,b].$ Let $c=a$ or $c=b$ and $\phi(x)\neq 0$ for all $x\in [a,b]\setminus \{c\}.$ If $0<\alpha<2$ and $\beta\geq \alpha+1,$ then by \eqref{MLAsym} we have
\begin{align*}
\int\limits_I\left|E_{\alpha,\beta}\left(i^\alpha\lambda \phi(x)\right)\right|dx \leq C\int\limits_I\frac{1}{1+\lambda|\phi(x)|}dx,\end{align*}
where $C$ is an arbitrary constant independent of $\phi$ and $\lambda.$

Let $c=a.$ Since $\phi\in C^1(I)$ and $|\phi'(x)|\geq 1,$ then
\begin{equation}\label{++}\begin{split}
\int\limits_I\left|E_{\alpha,\beta}\left(i^\alpha\lambda \phi(x)\right)\right|dx& \leq \frac{C}{\lambda}\int\limits_I\frac{1}{\phi'(x)}\frac{\lambda\phi'(x)}{1+\lambda|\phi(x)|}dx \leq \frac{C}{\lambda}\int\limits_I\frac{d(1+\lambda|\phi(x)|)}{1+\lambda|\phi(x)|}\\&\leq \frac{C}{\lambda}\log(1+\lambda|\phi(b)|)\leq \frac{M_1}{\lambda}\log(1+\lambda),\end{split}\end{equation} where $M_1$ is the arbitrary constant independent of $\lambda.$

Let $\phi(c)=0,\,c\in (a,b).$ Then
\begin{align*}
|I_{\alpha,\beta}(\lambda)|&\leq \left(\int\limits_a^c+\int\limits_c^b\right)\left|E_{\alpha,\beta}\left(i^\alpha\lambda \phi(x)\right)\right|dx.
\end{align*} Further, repeating the above calculations we have \eqref{2-1} for $k=1.$

Let $\lambda\geq 1$ and $k=2.$ We assume that \eqref{2-1} is true for $k=1$ and let $|\phi''(x)|\geq 1,$ for all $x\in I,$ we prove the estimate \eqref{2-1} for $k=2.$
Let $d\in [a,b]$ be a unique point where $|\phi'(d)|\leq |\phi'(x)|$ for all $x\in I.$ If $\phi'(d)=0,$ then we obtain $|\phi'(x)|\geq\epsilon$ on $I$ outside $(d-\epsilon, d+\epsilon).$
Further, we will write $\int\limits_a^bE_{\alpha,\beta}\left(i^\alpha\lambda \phi(x)\right)dx$ as
\begin{align*}\int\limits_a^bE_{\alpha,\beta}\left(i^\alpha\lambda \phi(x)\right)dx=\left(\int\limits_a^{c-\epsilon}+\int\limits_{c-\epsilon}^{c+\epsilon}+\int\limits_{c+\epsilon}^b\right) E_{\alpha,\beta}\left(i^\alpha\lambda \phi(x)\right)dx.\end{align*} As $|\phi'(x)|\geq\epsilon$ on $I$ outside $(d-\epsilon, d+\epsilon),$ then by the case $k=1,$ we have
\begin{align*}\left|\int\limits_a^{c-\epsilon}E_{\alpha,\beta}\left(i^\alpha\lambda \phi(x)\right)dx\right|&\leq M_2\frac{\log(1+\lambda)}{(\epsilon\lambda)},\end{align*} and \begin{align*}\left|\int\limits_{c+\epsilon}^bE_{\alpha,\beta}\left(i^\alpha\lambda \phi(x)\right)dx\right| &\leq M_2\frac{\log(1+\lambda)}{(\epsilon\lambda)}.\end{align*} As $$\left|\int\limits_{c-\epsilon}^{c+\epsilon}E_{\alpha,\beta}\left(i^\alpha\lambda \phi(x)\right)dx\right|\leq 2\epsilon,$$ we have $$\left|\int\limits_a^bE_{\alpha,\beta}\left(i^\alpha\lambda \phi(x)\right)dx\right|\leq 2M_2\frac{\log(1+\lambda)}{(\epsilon\lambda)}+2\epsilon.$$
Taking $\epsilon=\lambda^{-\frac{1}{2}}\log^{\frac{1}{2}}(1+\lambda)$ we obtain the estimate \eqref{2-1} for $k=2$, which proves the result.

Let us prove the estimate \eqref{2-1} by the induction method for $k\geq 2.$ We assume that \eqref{2-1} is true for $k\geq 2.$ And assuming $|\phi^{(k+1)}(x)|\geq 1,$ for all $x\in I,$ we prove the estimate \eqref{2-1} for $k + 1.$
Let $d\in [a,b]$ be a unique point where $|\phi^{(k)}(d)|\leq |\phi^{(k)}(x)|$ for all $x\in I.$ If $\phi^{(k)}(d)=0,$ then we obtain $|\phi^{(k)}(x)|\geq\epsilon$ on $I$ outside $(d-\epsilon, d+\epsilon).$
Further, we will write $\int\limits_a^bE_{\alpha,\beta}\left(i^\alpha\lambda \phi(x)\right)dx$ as
\begin{align*}\int\limits_a^bE_{\alpha,\beta}\left(i^\alpha\lambda \phi(x)\right)dx=\left(\int\limits_a^{d-\epsilon}+\int\limits_{d-\epsilon}^{d+\epsilon}+\int\limits_{d+\epsilon}^b\right) E_{\alpha,\beta}\left(i^\alpha\lambda \phi(x)\right)dx.\end{align*} By inductive hypothesis
\begin{align*}\left|\int\limits_a^{d-\epsilon}E_{\alpha,\beta}\left(i^\alpha\lambda \phi(x)\right)dx\right|&\leq M_k(\epsilon\lambda)^{-\frac{1}{k}}\log^{\frac{1}{k}}(1+\lambda),\end{align*} and \begin{align*}\left|\int\limits_{d+\epsilon}^bE_{\alpha,\beta}\left(i^\alpha\lambda \phi(x)\right)dx\right| &\leq M_k(\epsilon\lambda)^{-\frac{1}{k}}\log^{\frac{1}{k}}(1+\lambda).\end{align*} As $$\left|\int\limits_{d-\epsilon}^{d+\epsilon}E_{\alpha,\beta}\left(i^\alpha\lambda \phi(x)\right)dx\right|\leq 2\epsilon,$$ we have $$\left|\int\limits_a^bE_{\alpha,\beta}\left(i^\alpha\lambda \phi(x)\right)dx\right|\leq 2M_k(\epsilon\lambda)^{-\frac{1}{k}}\log^{\frac{1}{k}}(1+\lambda)+2\epsilon.$$
Taking $\epsilon=\lambda^{-\frac{1}{k+1}}\log^{\frac{1}{k+1}}(1+\lambda)$ we obtain the estimate \eqref{2-1} for $k+1$, which proves the result.\\
{\bf Proof of (ii).} Let $k=1.$
Since $|\phi'(x)|\geq 1$ for all $x\in I,$ then $|\phi|$ is monotonic, and it can have one zero at $c\in[a,b].$ Let $c=a$ or $c=b$ and $\phi(x)\neq 0$ for all $x\in [a,b]\setminus \{c\}.$ If $0<\alpha<2$ and $1<\beta< \alpha+1,$ then by \eqref{MLAsym} we have
\begin{align*}
\left|\int\limits_IE_{\alpha,\beta}\left(i^\alpha\lambda \phi(x)\right)dx\right|& \leq \int\limits_I\left|E_{\alpha,\beta}\left(i^\alpha\lambda \phi(x)\right)\right|dx\\& \leq C\int\limits_I\frac{1}{(1+\lambda|\phi(x)|)^{\frac{\beta-1}{\alpha}}}dx,\end{align*}
where $C$ is an arbitrary constant independent of $\phi$ and $\lambda.$
Let $c=a,$ then by $\phi\in C^1(I)$ and $|\phi'(x)|\geq 1,$ we have
\begin{align*}
\left|\int\limits_IE_{\alpha,\beta}\left(i^\alpha\lambda \phi(x)\right)dx\right|& \leq \frac{C}{\lambda}\int\limits_I\frac{1}{\phi'(x)}\frac{\lambda\phi'(x)}{(1+\lambda|\phi(x)|)^{\frac{\beta-1}{\alpha}}}dx \leq \frac{C}{\lambda}\int\limits_I\frac{d(1+\lambda|\phi(x)|)}{(1+\lambda|\phi(x)|)^{\frac{\beta-1}{\alpha}}}\\&\leq \frac{C}{\lambda}\left[(1+\lambda|\phi(b)|)^{\frac{\alpha+1-\beta}{\alpha}}-1\right] \leq \frac{C}{\lambda}\left[(1+\lambda|\phi(b)|)^{\frac{\alpha+1-\beta}{\alpha}}+1\right]\\&\leq \frac{M_2}{\lambda}(1+\lambda)^{\frac{\alpha+1-\beta}{\alpha}},\end{align*} where $M_2$ is the arbitrary constant independent of $\lambda.$

Let $\phi(c)=0,\,c\in (a,b).$ Then
\begin{align*}
|I_{\alpha,\beta}(\lambda)|&\leq \left(\int\limits_a^c+\int\limits_c^b\right)\left|E_{\alpha,\beta}\left(i^\alpha\lambda \phi(x)\right)\right|dx.
\end{align*} Further, repeating the above calculations we have \eqref{2-1*} for $k=1.$

Let $\lambda\geq 1,$ $0<\alpha<2$ and $1<\beta< \alpha+1,$ $k= 2.$ We assume that \eqref{2-1*} holds for $k=1.$ Assuming $|\phi''(x)|\geq 1,$ for all $x\in I,$ we prove the estimate \eqref{2-1*} for $k=2.$
Let $d\in [a,b]$ be a unique point where $|\phi'(d)|\leq |\phi'(x)|$ for all $x\in I.$ If $\phi'(d)=0,$ then we obtain $|\phi'(x)|\geq\epsilon$ on $I$ outside $(d-\epsilon, d+\epsilon).$
Further, we will write $\int\limits_a^bE_{\alpha,\beta}\left(i^\alpha\lambda \phi(x)\right)dx$ as
\begin{align*}\int\limits_a^bE_{\alpha,\beta}\left(i^\alpha\lambda \phi(x)\right)dx=\left(\int\limits_a^{d-\epsilon}+\int\limits_{d-\epsilon}^{d+\epsilon}+\int\limits_{d+\epsilon}^b\right) E_{\alpha,\beta}\left(i^\alpha\lambda \phi(x)\right)dx.\end{align*} As $|\phi'(x)|\geq\epsilon$ on $I$ outside $(d-\epsilon, d+\epsilon),$ by the case $k=1,$ we have
\begin{align*}\left|\int\limits_a^{d-\epsilon}E_{\alpha,\beta}\left(i^\alpha\lambda \phi(x)\right)dx\right|&\leq M_1\frac{(1+\lambda)^{\frac{\alpha+1-\beta}{\alpha}}}{(\epsilon\lambda)},\end{align*} and \begin{align*}\left|\int\limits_{d+\epsilon}^bE_{\alpha,\beta}\left(i^\alpha\lambda \phi(x)\right)dx\right| &\leq M_1\frac{(1+\lambda)^{\frac{\alpha+1-\beta}{\alpha}}}{(\epsilon\lambda)}.\end{align*} As $$\left|\int\limits_{d-\epsilon}^{d+\epsilon}E_{\alpha,\beta}\left(i^\alpha\lambda \phi(x)\right)dx\right|\leq 2\epsilon,$$ we have $$\left|\int\limits_a^bE_{\alpha,\beta}\left(i^\alpha\lambda \phi(x)\right)dx\right|\leq 2M_1\frac{(1+\lambda)^{\frac{\alpha+1-\beta}{\alpha}}}{(\epsilon\lambda)}+2\epsilon.$$
Taking $\epsilon=\lambda^{-\frac{1}{2}}(1+\lambda)^{\frac{1+\alpha-\beta}{2\alpha}}$ we obtain the estimate \eqref{2-1*} for $k=2$, which proves the result. The cases when $c = a$ or $c = b$ can be proved similarly.

Let us prove the estimate \eqref{2-1*} by induction method on $k\geq 2.$ We assume that \eqref{2-1*} holds for $k\geq 2.$ Assuming $|\phi^{(k+1)}(x)|\geq 1,$ for all $x\in I,$ we prove the estimate \eqref{2-1*} for $k + 1.$
Let $d\in [a,b]$ be a unique point where $|\phi^{(k)}(d)|\leq |\phi^{(k)}(x)|$ for all $x\in I.$ If $\phi^{(k)}(d)=0,$ then we obtain $|\phi^{(k)}(x)|\geq\epsilon$ on $I$ outside $(d-\epsilon, d+\epsilon).$
Further, we will write $\int\limits_a^bE_{\alpha,\beta}\left(i^\alpha\lambda \phi(x)\right)dx$ as
\begin{align*}\int\limits_a^bE_{\alpha,\beta}\left(i^\alpha\lambda \phi(x)\right)dx=\left(\int\limits_a^{d-\epsilon}+\int\limits_{d-\epsilon}^{d+\epsilon}+\int\limits_{d+\epsilon}^b\right) E_{\alpha,\beta}\left(i^\alpha\lambda \phi(x)\right)dx.\end{align*} By inductive hypothesis
\begin{align*}\left|\int\limits_a^{d-\epsilon}E_{\alpha,\beta}\left(i^\alpha\lambda \phi(x)\right)dx\right|&\leq M_k(\epsilon\lambda)^{\frac{1}{k}}(1+\lambda)^{\frac{\alpha+1-\beta}{\alpha k}},\end{align*} and \begin{align*}\left|\int\limits_{d+\epsilon}^bE_{\alpha,\beta}\left(i^\alpha\lambda \phi(x)\right)dx\right| &\leq M_k(\epsilon\lambda)^{\frac{1}{k}}(1+\lambda)^{\frac{\alpha+1-\beta}{\alpha k}}.\end{align*} As $$\left|\int\limits_{d-\epsilon}^{d+\epsilon}E_{\alpha,\beta}\left(i^\alpha\lambda \phi(x)\right)dx\right|\leq 2\epsilon,$$ we have $$\left|\int\limits_a^bE_{\alpha,\beta}\left(i^\alpha\lambda \phi(x)\right)dx\right|\leq 2M_k(\epsilon\lambda)^{\frac{1}{k}}(1+\lambda)^{\frac{\alpha+1-\beta}{\alpha k}}+2\epsilon.$$
Taking $\epsilon=\lambda^{\frac{1}{k+1}}(1+\lambda)^{\frac{\alpha+1-\beta}{\alpha (k+1)}}$ we obtain the estimate \eqref{2-1*} for $k+1$, which proves the result. The cases when $c = a$ or $c = b$ can be proved similarly.
\end{proof}
\begin{theorem}\label{theor2.1} Let $-\infty< a<b<+\infty$ and $I=[a,b]\subset \mathbb{R}.$ Let  $0<\alpha<2,\,\beta>1,$ $\phi$ is a real-valued function such that $\phi\in C^k(I),\,k\geq 1$ and let $\psi\in C^1(I).$ If
$|\phi^{(k)}(x)|\geq 1$ for all $x\in I,$ then
\begin{description}
  \item[(i)] for $0<\alpha<2$ and $\beta\geq \alpha+1$ we have
\begin{equation}\label{2-2}
\left|I_{\alpha,\beta}\left(\lambda \right)\right|\leq M_k\left[|\psi(b)|+\int\limits_I|\psi'(x)|dx\right]\lambda^{-\frac{1}{k}}\log^{\frac{1}{k}}(1+\lambda),\,\,\,\lambda\geq 1,
\end{equation} where $M_k$ does not depend on $\lambda;$
  \item[(ii)] for $0<\alpha<2$ and $1<\beta< \alpha+1$ we have
\begin{equation}\label{2-2*}
\left|I_{\alpha,\beta}\left(\lambda \right)\right|\leq M_k\left[|\psi(b)|+\int\limits_I|\psi'(x)|dx\right]\lambda^{-\frac{1}{k}}(1+\lambda)^{\frac{\alpha+1-\beta}{\alpha k}},\,\,\,\lambda\geq 1,
\end{equation} where $M_k$ does not depend on $\lambda.$
\end{description}
\end{theorem}
\begin{proof} We write \eqref{1*} as $$\int\limits_I E'(x)\psi(x)dx,$$ where $$E(x)=\int\limits_a^xE_{\alpha,\beta}\left(i^\alpha\lambda \phi(s)\right)ds.$$ Let $0<\alpha<2$ and $\beta\geq \alpha+1.$ Integrating by parts and applying the estimate of part (i) of Theorem \ref{th2.1} we obtain
\begin{equation*}|I_{\alpha,\beta}(\lambda)|\leq M_k\frac{\log(1+\lambda)}{(1+\lambda)^{1/k}}\left[|\psi(b)|+\int\limits_I|\psi'(x)|dx\right].
\end{equation*} The case (ii) can be proved similarly by applying results of part (ii) of Theorem \ref{th2.1}.
\end{proof}

Below we are interested in a particular case of the integral \eqref{1*}, when $0<\alpha<2,\,\beta=\alpha,$ that is, the integral $I_{\alpha,\alpha}(\lambda).$ For smoother $\phi$ and $\psi$ we get a better estimate than \eqref{2-1*} and \eqref{2-2*}.

\begin{theorem}\label{th1.3} Let $-\infty< a<b<+\infty$ and $I=[a,b]\subset \mathbb{R}.$ Let  $0<\alpha<2$ and let $\phi$ be a real-valued function such that $\phi\in C^k(I),\,k\geq 1.$ Let $|\phi^{(k)}(x)|\geq 1$ for all $x\in I,$ then
\begin{equation}\label{2.1}
\left|\int\limits_IE_{\alpha,\alpha}\left(i^\alpha\lambda \phi(x)\right)dx\right|\leq M_k\lambda^{-\frac{1}{k}},\,\,\,\lambda\geq 1,
\end{equation} for $k=1$ and $\phi'$ is monotonic, or $k\geq 2$. Here $M_k$ does not depend on $\lambda.$
\end{theorem}

We note that the classical van der Corput lemma \eqref{vdC} is covered by \eqref{2.1} with $\alpha=1$.
\begin{proof}
First we will prove the case $k=1.$
Let $0<\alpha<2,$ $\lambda\geq1$ and let $\phi$ has one zero $c\in[a,b].$ Let us consider the integral $$\int\limits_IE_{\alpha,\alpha}\left(i^\alpha\lambda \phi(x)\right)dx.$$ Then integrating by parts gives
\begin{align*}
\int\limits_IE_{\alpha,\alpha}\left(i^\alpha\lambda \phi(x)\right)dx&=\frac{\alpha}{i^\alpha\lambda}\int\limits_I\frac{d}{d\phi(x)}\left(E_{\alpha,1}\left(i^\alpha\lambda \phi(x)\right)\right)dx\\&=\frac{\alpha}{i^\alpha\lambda}\int\limits_I\frac{1}{\phi'(x)}\frac{d}{dx}\left(E_{\alpha,1}\left(i^\alpha\lambda \phi(x)\right)\right)dx\\&=\frac{\alpha}{i^\alpha\lambda}E_{\alpha,1}\left(i^\alpha\lambda \phi(b)\right)\frac{1}{\phi'(b)}
-\frac{\alpha}{i^\alpha\lambda}E_{\alpha,1}\left(i^\alpha\lambda \phi(a)\right)\frac{1}{\phi'(a)}\\&-\frac{\alpha}{i^\alpha\lambda}\int\limits_IE_{\alpha,1}\left(i^\alpha\lambda \phi(x)\right)\frac{d}{dx}\left(\frac{1}{\phi'(x)}\right)dx,
\end{align*} thanks to property $\frac{d}{dz}E_{\alpha,1}(z)=\frac{1}{\alpha}E_{\alpha,\alpha}(z)$ and $|\phi'(x)|\geq 1,\,x\in I.$
Then we have
\begin{equation}\label{01}\begin{split}
\left|\int\limits_IE_{\alpha,\alpha}\left(i^\alpha\lambda \phi(x)\right)dx\right|&\leq\frac{\alpha}{\lambda}\int\limits_I\left|E_{\alpha,1}\left(i^\alpha\lambda \phi(x)\right)\right|\left|\frac{d}{dx}\left(\frac{1}{\phi'(x)}\right)\right|dx\\& +\frac{\alpha}{\lambda}\left|E_{\alpha,1}\left(i^\alpha\lambda \phi(b)\right)\right|\frac{1}{|\phi'(b)|}
\\&+\frac{\alpha}{\lambda}\left|E_{\alpha,1}\left(i^\alpha\lambda \phi(a)\right)\right|\frac{1}{|\phi'(a)|}.
\end{split}\end{equation}
As $\phi'$ is monotonic and $\phi'(x)\geq 1$ for all $x\in[a,b],$ then $\frac{1}{\phi'}$ is also monotonic,
and $\frac{d}{dx}\frac{1}{\phi'(x)}$ has a fixed sign. Hence estimate \eqref{MLAsym} and $\phi(c)=0,\,c\in[a,b]$ implies
\begin{align*}
\left|\int\limits_IE_{\alpha,\alpha}\left(i^\alpha\lambda \phi(x)\right)dx\right|&\leq\frac{C\alpha}{\lambda}\int\limits_I(1+\lambda |\phi(x)|)^{\frac{1-\alpha}{\alpha}}\left|\frac{d}{dx}\left(\frac{1}{\phi'(x)}\right)\right|dx \\& +\frac{C\alpha}{\lambda}(1+\lambda |\phi(b)|)^{\frac{1-\alpha}{\alpha}}+\frac{C\alpha}{\lambda}(1+\lambda |\phi(a)|)^{\frac{1-\alpha}{\alpha}} \\&\stackrel{|\phi(x)|\geq 0}\leq\frac{C\alpha}{\lambda}\int\limits_I\left|\frac{d}{dx}\left(\frac{1}{\phi'(x)}\right)\right|dx +\frac{2C\alpha}{\lambda}\\& \leq\frac{C\alpha}{\lambda}\left|\int\limits_I\frac{d}{dx}\left(\frac{1}{\phi'(x)}\right)dx\right| +\frac{2C\alpha}{\lambda}
\\&\leq \frac{C\alpha}{\lambda}\left[2+\frac{1}{|\phi'(b)|}+\frac{1}{|\phi'(b)|}\right]\leq \frac{M_1}{\lambda},
\end{align*} thanks to fixed sign of $\frac{d}{dx}\frac{1}{\phi'(x)}.$ Here $M_1$ does not depended on $\lambda.$

We prove \eqref{2.1} for $k=2.$ Let for $d\in [a,b]$ satisfies $|\phi'(d)|\leq |\phi'(x)|\,\, \textrm{for all}\,\, x\in [a,b].$ Then $|\phi'(x)|\geq\epsilon$ on $[a,b]\setminus(d-\epsilon, d+\epsilon).$ Hence
$$\int\limits_IE_{\alpha,\alpha}\left(i^\alpha\lambda \phi(x)\right)dx=\left(\int\limits_a^{d-\epsilon}+\int\limits_{d-\epsilon}^{d+\epsilon}+ \int\limits_{d+\epsilon}^b\right)E_{\alpha,\alpha}\left(i^\alpha\lambda \phi(x)\right)dx.$$ Then $|\phi'(x)|\geq\epsilon$ on $[a,b]\setminus(d-\epsilon, d+\epsilon),$ imples
$$\left|\int\limits_a^{d-\epsilon}E_{\alpha,\alpha}\left(i^\alpha\lambda \phi(x)\right)dx\right|\leq M_1(\epsilon\lambda)^{-1},$$ and $$\left|\int\limits_{d+\epsilon}^bE_{\alpha,\alpha}\left(i^\alpha\lambda \phi(x)\right)dx\right| \leq  M_1(\epsilon\lambda)^{-1}.$$ As $$\left|\int\limits_{d-\epsilon}^{d+\epsilon}E_{\alpha,\alpha}\left(i^\alpha\lambda \phi(x)\right)dx\right|\leq 2\epsilon,$$ we have $$\left|\int\limits_a^bE_{\alpha,\alpha}\left(i^\alpha\lambda \phi(x)\right)dx\right|\leq 2M_1(\epsilon\lambda)^{-1}+2\epsilon.$$
Taking $\epsilon=\lambda^{-\frac{1}{2}}$ we obtain the estimate \eqref{2.1} for $k=2$.

We prove the case $k\geq 2$ by induction method. Let \eqref{2.1} is true for $k,$ and suppose $|\phi^{(k+1)}(x)|\geq 1,$ for all $x\in [a,b],$ we prove \eqref{2.1} for $k + 1.$

Let for $d\in [a,b]$ satisfies $|\phi^{(k)}(d)|\leq |\phi^{(k)}(x)|\,\, \textrm{for all}\,\, x\in [a,b].$ Then $|\phi^{(k)}(x)|\geq\epsilon$ on $[a,b]\setminus(c-\epsilon, c+\epsilon).$ Therefore
$$\int\limits_IE_{\alpha,\alpha}\left(i^\alpha\lambda \phi(x)\right)dx=\left(\int\limits_a^{d-\epsilon}+\int\limits_{d-\epsilon}^{d+\epsilon}+ \int\limits_{d+\epsilon}^b\right)E_{\alpha,\alpha}\left(i^\alpha\lambda \phi(x)\right)dx.$$ By inductive hypothesis
$$\left|\int\limits_a^{d-\epsilon}E_{\alpha,\alpha}\left(i^\alpha\lambda \phi(x)\right)dx\right|\leq M_k(\epsilon\lambda)^{-\frac{1}{k}},$$ and $$\left|\int\limits_{d+\epsilon}^bE_{\alpha,\alpha}\left(i^\alpha\lambda \phi(x)\right)dx\right| \leq  M_k(\epsilon\lambda)^{-\frac{1}{k}}.$$ As $$\left|\int\limits_{d-\epsilon}^{d+\epsilon}E_{\alpha,\alpha}\left(i^\alpha\lambda \phi(x)\right)dx\right|\leq 2\epsilon,$$ we have $$\left|\int\limits_a^bE_{\alpha,\alpha}\left(i^\alpha\lambda \phi(x)\right)dx\right|\leq 2M_k(\epsilon\lambda)^{-\frac{1}{k}}+2\epsilon.$$
Taking $\epsilon=\lambda^{-\frac{1}{k+1}}$ we obtain the estimate \eqref{2.1} for $k+1$.
\end{proof}
Below we show that if $\phi'$ is not monotonic, then to obtain estimate \eqref{2.1} when $k=1,$ it is necessary to increase the smoothness of function $\phi.$
\begin{theorem}\label{th3} Let $-\infty< a<b<+\infty$ and $I=[a,b]\subset \mathbb{R}.$ Let  $0<\alpha<2$ and let $\phi$ be a real-valued function such that $\phi\in C^2(I).$ Let $|\phi'(x)|\geq 1$ for all $x\in I,$ then
\begin{equation*}
\left|\int\limits_IE_{\alpha,\alpha}\left(i^\alpha\lambda \phi(x)\right)dx\right|\leq M\lambda^{-1},\,\,\,\lambda\geq 1,
\end{equation*} where $M$ does not depend on $\lambda.$
\end{theorem}
\begin{proof} Suppose that $\phi\in C^2(I)$ and $|\phi'(x)|\geq 1$ for all $x\in I,$ then from \eqref{01} we have
\begin{equation*}\begin{split}
\left|\int\limits_IE_{\alpha,\alpha}\left(i^\alpha\lambda \phi(x)\right)dx\right|&\leq\frac{\alpha}{\lambda}\int\limits_I\left|E_{\alpha,1}\left(i^\alpha\lambda \phi(x)\right)\right|\left|\frac{d}{dx}\left(\frac{1}{\phi'(x)}\right)\right|dx\\& +\frac{\alpha}{\lambda}\left|E_{\alpha,1}\left(i^\alpha\lambda \phi(b)\right)\right|
\\&+\frac{\alpha}{\lambda}\left|E_{\alpha,1}\left(i^\alpha\lambda \phi(a)\right)\right|.
\end{split}\end{equation*} Since $\phi\in C^2(I)$ and $|\phi'(x)|\geq 1$ for all $x\in I,$ then the function $\frac{d}{dx}\left(\frac{1}{\phi'(x)}\right)$ will be continuous and bounded, and therefore by \eqref{MLAsym}  we have
\begin{equation*}\begin{split}
\left|\int\limits_IE_{\alpha,\alpha}\left(i^\alpha\lambda \phi(x)\right)dx\right|&\leq \frac{C\alpha}{\lambda}\left[\int\limits_I\left|\frac{d}{dx}\left(\frac{1}{\phi'(x)}\right)\right|dx+2\right]\\& \leq \frac{C\alpha}{\lambda}(M_1(b-a)+2)\leq \frac{M}{\lambda},
\end{split}\end{equation*} where $M_1=\left\|\frac{d}{dx}\left(\frac{1}{\phi'(x)}\right)\right\|_{L^\infty(I)},$ $C\geq\left|E_{\alpha,1}\left(z\right)\right|$ and $M$ is a constant independent of $\lambda.$
\end{proof}
\begin{theorem}\label{th1.4} Let  $0<\alpha<2$ and let $\phi$ be a real-valued function such that $\phi\in C^k(I),\,k\geq 1.$ Let $\psi\in C^1(I)$ and $|\phi^{(k)}(x)|\geq 1$ for all $x\in I,$ then
\begin{equation*}
\left|I_{\alpha,\alpha}\left(\lambda\right)\right|\leq M_k\left[|\psi(b)|+\int\limits_I|\psi'(x)|dx\right]\lambda^{-1/k},\,\,\,\lambda\geq 1,
\end{equation*} for $k=1$ and $\phi'$ is monotonic, or $k\geq 2$. Here $M_k$ does not depend on $\lambda.$
\end{theorem}
Theorem \ref{th1.4} can be proved similarly as Theorem \ref{th2.1}.

The case of $\alpha=1$ corresponds to the classical van der Corput lemma \eqref{vdC}.

Also, for $k=1,$ Theorem \ref{th1.4} holds if we replace the condition that $\phi'$ is monotonic by $\phi\in C^2(I).$

\section{Applications}
\label{SEC:apps}
In this section we give some applications of van der Corput lemmas involving Mittag-Leffler function.
\subsection{Applications to the fractional evolution equations}
\subsubsection{Decay estimates for the time-fractional Klein-Gordon equation}
Consider the time-fractional Klein-Gordon equation
\begin{equation}\label{Shr1}
{D}^\alpha_{0+,t}u(t,x)+i^\alpha u_{xx}(t,x)-i^\alpha\mu u(t,x)=0,\,\,t>0,\,x\in\mathbb{R},
\end{equation}
with initial data
\begin{equation}\label{Shr2}
I^{2-\alpha}_{0+,t}u(0,x)=0,\,x\in \mathbb{R},\end{equation} \begin{equation}\label{Shr2*}\partial_tI^{2-\alpha}_{0+,t}u(0,x)=\psi(x),\,x\in \mathbb{R},
\end{equation}
where $\mu>0,$ and, $$I^{\alpha}_{0+,t}u(t,x)=\frac{1}{\Gamma(\alpha)}\int\limits_0^t \left(t-s\right)^{\alpha-1}u(s,x)ds$$ and $${D}^{\alpha}_{0+,t}u(t,x)=\frac{1}{\Gamma(2-\alpha)}\partial^2_t\int\limits_0^t \left(t-s\right)^{1-\alpha}u(s,x)ds$$ is the Riemann-Liouville fractional integral and derivative of order $1<\alpha\leq 2.$

If $\alpha = 2,$ then from \eqref{Shr1} we obtain a classical Klein-Gordon equation.

Applying the Fourier transform $\mathcal{F}$ to problem \eqref{Shr1}-\eqref{Shr2*} with respect to space variable $x$ yields
\begin{equation}\label{Shr3}{D}_{0+,t}^{\alpha } \hat{u}
\left( t, \xi\right) - i^\alpha(\xi^{2}+\mu) \hat{u}\left( t, \xi \right) = 0,\, t >0,\,\xi\in \mathbb{R},\end{equation} \begin{equation}\label{Shr4*}I^{2-\alpha}_{0+,t}\hat{u} \left(0,\xi\right)=0,\,\,\xi\in \mathbb{R},\end{equation} \begin{equation}\label{Shr4}\partial_tI^{2-\alpha}_{0+,t}\hat{u}_t \left(0,\xi\right) =\hat{\psi}(\xi),\,\,\xi\in \mathbb{R},\end{equation} due to $\mathcal{F}\left\{u_{xx}(t,x)\right\}=-\xi^{2}\hat{u}(t,\xi).$ The general solution of equation
\eqref{Shr3} can be represented as
\begin{equation*}\hat{u} \left(
t, \xi\right) = C_1(\xi) t^{\alpha-1}E_{\alpha, \alpha} \left( {i^\alpha(\xi^{2}+\mu) t^{\alpha}  }\right)+C_2(\xi) t^{\alpha-2}E_{\alpha, \alpha-1} \left( {i^\alpha(\xi^{2}+\mu) t^{\alpha}  }\right),\end{equation*} where $C_1(\xi)$ and $C_2(\xi)$ are unknown coefficients. Then by initial conditions \eqref{Shr4*}-\eqref{Shr4} we have
\begin{equation*}\hat{u} \left(
t, \xi\right) = \hat{\psi}(\xi) t^{\alpha-1}E_{\alpha, \alpha} \left( {i^\alpha(\xi^{2}+\mu) t^{\alpha}  }\right).\end{equation*}
By applying the inverse Fourier transform $\mathcal{F}^{-1}$ we have
\begin{equation}\label{Shr5}
u(t,x)=\int\limits_{\mathbb{R}}e^{ix\xi}t^{\alpha-1}E_{\alpha,\alpha}\left(i^\alpha(\xi^2+\mu) t^\alpha\right)\hat{\psi}(\xi)d\xi,
\end{equation} where $\hat{\psi}(\xi)=\frac{1}{\pi}\int\limits_{\mathbb{R}}e^{-iy\xi}\psi(y)dy.$

Suppose that $\psi\in L^1(\mathbb{R})$ and $\hat{\psi}\in L^1(\mathbb{R}).$ As $\inf\limits_{\xi\in \mathbb{R}}(\xi^2+\mu)=\mu>0,$ then using Theorem \ref{th1} (ii) we obtain the dispersive estimate
$$\|u(t,\cdot)\|_{L^\infty(\mathbb{R})}\leq Ct^{\alpha-1}(1+t^\alpha)^{\frac{1-\alpha}{\alpha}}\|\hat{\psi}\|_{L^1(\mathbb{R})}\leq Ct^{\alpha-1}(1+t)^{1-\alpha}\|\hat{\psi}\|_{L^1(\mathbb{R})},\,t> 0.$$
\subsubsection{Decay estimates for the time-fractional Schr\"{o}dinger equation}
Consider the time-fractional Schr\"{o}dinger equation
\begin{equation}\label{Schr1}
\mathcal{D}^\alpha_{0+,t}u(t,x)+i^\alpha u_{xx}(t,x)-i^\alpha\mu u(t,x)=t^{-\gamma}\psi(x),\,\,t>0,\,x\in\mathbb{R},
\end{equation}
with Cauchy data
\begin{equation}\label{Schr2}
u(0,x)=0,\,x\in \mathbb{R},\end{equation}
where $\mu>0,$ $0\leq\gamma<\alpha,$ and, $$\mathcal{D}^{\alpha}_{0+,t}u(t,x)=\frac{1}{\Gamma(1-\alpha)}\int\limits_0^t \left(t-s\right)^{1-\alpha}\partial_su(s,x)ds$$ is the Caputo fractional derivative of order $0<\alpha\leq 1.$

When $\alpha = 1,$ we have the classical Schr\"{o}dinger equation.

Applying the Fourier transform $\mathcal{F}$ to problem \eqref{Schr1}-\eqref{Schr2} with respect to space variable $x$ yields
\begin{equation}\label{Schr3}{D}_{0+,t}^{\alpha } \hat{u}
\left( t, \xi\right) - i^\alpha(\xi^{2}+\mu) \hat{u}\left( t, \xi \right) = t^{-\gamma}\hat{\psi}(\xi),\, t >0,\,\xi\in \mathbb{R},\end{equation} \begin{equation}\label{Schr4}\hat{u} \left(0,\xi\right)=0,\,\,\xi\in \mathbb{R}.\end{equation} The solution of problem
\eqref{Schr3}-\eqref{Schr4} can be represented as
\begin{equation*}\hat{u} \left(
t, \xi\right) = \int\limits_0^t (t-s)^{\alpha-1}E_{\alpha, \alpha} \left( {i^\alpha(\xi^{2}+\mu) (t-s)^{\alpha}  }\right)s^{-\gamma}\hat{\psi}(\xi)ds.\end{equation*} Calculating the above integral we obtain
\begin{equation*}\hat{u} \left(
t, \xi\right) = \Gamma(1-\gamma)\hat{\psi}(\xi)t^{\alpha-\gamma}E_{\alpha, \alpha-\gamma+1} \left( {i^\alpha(\xi^{2}+\mu) t^{\alpha}  }\right).\end{equation*}
By applying the inverse Fourier transform $\mathcal{F}^{-1}$ we have
\begin{equation}\label{Shr5}
u(t,x)=\Gamma(1-\gamma)\int\limits_{\mathbb{R}}e^{ix\xi}t^{\alpha-\gamma}E_{\alpha, \alpha-\gamma+1} \left( {i^\alpha(\xi^{2}+\mu) t^{\alpha}  }\right)\hat{\psi}(\xi)d\xi,
\end{equation} where $\hat{\psi}(\xi)=\frac{1}{\pi}\int\limits_{\mathbb{R}}e^{-iy\xi}\psi(y)dy.$

Suppose that $\psi\in L^1(\mathbb{R})$ and $\hat{\psi}\in L^1(\mathbb{R}).$ As $\inf\limits_{\xi\in \mathbb{R}}(\xi^2+\mu)=\mu>0$ and $\alpha-\gamma+1<\alpha+1,$ then using Theorem \ref{th1} (ii) we obtain the estimate
$$\|u(t,\cdot)\|_{L^\infty(\mathbb{R})}\leq Ct^{\alpha-\gamma}(1+t^\alpha)^{\frac{\gamma-\alpha}{\alpha}}\|\hat{\psi}\|_{L^1(\mathbb{R})} \leq Ct^{\alpha-\gamma}(1+t)^{{\gamma-\alpha}}\|\hat{\psi}\|_{L^1(\mathbb{R})},\,t> 0.$$

\subsection{Generalised Riemann-Lebesgue lemma}
The Riemann-Lebesgue lemma is the classical result of harmonic and asymptotic analysis. The simplest form of the Riemann-Lebesgue lemma states that for a function $f\in L^1(\mathbb{R})$ we have
\begin{equation}\label{RiemLeb1}
\lim\limits_{k\rightarrow\infty}\int\limits_\mathbb{R}e^{ikx}f(x)dx=0.
\end{equation}
If $f\in C^1([a,b]),$ then
\begin{equation}\label{RiemLeb2}
\int\limits_a^be^{ikx}f(x)dx=\mathcal{O}\left(\frac{1}{k}\right),\,\,\,\textrm{at}\,\,\,k\rightarrow\infty.
\end{equation}
As a generalisation of Fourier integral we consider the following integral
\begin{equation}\label{RiemLeb3}\int\limits_DE_{\alpha,\beta}\left(i^\alpha k x\right)f(x)dx,\end{equation}
where $0<\alpha\leq 2,\,\beta\geq1,$ and $D\subseteq \mathbb{R}.$

Suppose that $D=\mathbb{R}$ and $f\in L^1(\mathbb{R}),$ then Theorem \ref{th2} yields
$$\lim\limits_{k\rightarrow\infty}\int\limits_\mathbb{R}E_{\alpha,1}\left(i^\alpha k x\right)f(x)dx=0.$$
For $\alpha=1$ this reduces to \eqref{RiemLeb1}.

If $f\in C^1([a,b]),$ then from the van der Corput lemmas we obtain
\begin{itemize}
  \item for $0<\alpha<2$ and $\beta\geq \alpha+1,$ by Theorem \ref{theor2.1} (i) we have
  \begin{equation*}
  \int\limits_a^bE_{\alpha,\beta}\left(i^\alpha k x\right)f(x)dx=\mathcal{O}\left(\frac{\log(1+k)}{k}\right);
  \end{equation*}
  \item for $0<\alpha<2$ and $1<\beta< \alpha+1,$ by Theorem \ref{theor2.1} (ii) we have
  \begin{equation*}
  \int\limits_a^bE_{\alpha,\beta}\left(i^\alpha k x\right)f(x)dx=\mathcal{O}\left(k^{-1}(1+k)^{\frac{\alpha+1-\beta}{\alpha}}\right);
  \end{equation*}
  \item for $0<\alpha<2$ and $\beta=\alpha,$ by Theorem \ref{th1.4} we have
  \begin{equation*}
  \int\limits_a^bE_{\alpha,\alpha}\left(i^\alpha k x\right)f(x)dx=\mathcal{O}\left(k^{-1}\right).
  \end{equation*}
\end{itemize}
\section*{Conclusion}
The main goal of the paper was to study van der Corput lemmas for the integral defined by
\begin{equation*}I_{\alpha,\beta}(\lambda)=\int\limits_a^bE_{\alpha,\beta}\left(i^\alpha\lambda \phi(x)\right)\psi(x)dx,\end{equation*}
with $0<\alpha\leq 2,\,\beta\geq 1.$

For $I_{\alpha,\beta}(\lambda),$ van der Corput type lemmas were obtained, for the following cases of parameters $\alpha$ and $\beta$:
\begin{itemize}
  \item $0<\alpha<2$ and $\beta\geq \alpha+1;$
  \item $0<\alpha<2$ and $1<\beta<\alpha+1;$
  \item $\alpha=2$ and $\beta\geq1;$
  \item $0<\alpha<2$ and $\beta=\alpha.$
\end{itemize}

As an immediate application of the obtained results, time-estimates of the solutions of time-fractional Klein-Gordon and Schr\"{o}dinger equations and generalisations of the Riemann-Lebesgue lemma were considered.


\begin{thebibliography}{ABGM15}
\bibitem[BG11]{Bourg} J. Bourgain, L. Guth. Bounds on oscillatory integral operators based on multilinear estimates. {\it Geom. Funct. Anal.} 21:6 (2011), 1239--1295.
\bibitem[CCW99]{CCW99} A. Carbery, M. Christ and J. Wright, Multidimensional van der Corput and sublevel set estimates, {\it J. Amer. Math. Soc.} (1999) 981--1015.
\bibitem[CLTT05]{Tao05} M. Christ, X. Li, T. Tao, C. Thiele. On multilinear oscillatory integrals, nonsingular and singular. {\it Duke Math. J.}, 130:2 (2005), 321--351.
\bibitem[DX08]{Dong} J. Dong, M. Xu, Space-time fractional Schr\"{o}dinger equation with time-independent potentials, {\it J. Math. Anal. Appl.}, 344 (2008), 1005--1017.
\bibitem[GKMR14]{GorKMR} R. Gorenflo, A. A. Kilbas, F. Mainardi, S. V. Rogosin. {\it Mittag-Leffler Functions, Related Topics and Applications.} Springer Monographs in Mathematics. Springer, Heidelberg, 2014.
\bibitem[Gr19]{Grande} R. Grande, Space-time fractional nonlinear Schr\"{o}dinger equation, {\it SIAM J. Math. Anal.}, 51:5 (2019), 4172--4212.
\bibitem[Gr05]{Gr05} M. Greenblatt. Sharp  estimates for one-dimensional oscillatory integral operators with phase. {\it Amer. J. Math.}, 127:3 (2005), 659--695.
\bibitem[GPT07]{Green} A. Greenleaf, M. Pramanik, and W. Tang, Oscillatory integral operators with homogeneous polynomial phases in several variables, {\it J. Funct. Anal.} 244:2 (2007), 444--487.
\bibitem[KR07]{KR07} I. Kamotski and M. Ruzhansky, Regularity properties, representation of solutions and spectral asymptotics of systems with multiplicities, {\it Comm. Partial Differential Equations}, 32 (2007) 1--35.
\bibitem[KST06]{Kilbas} A. A. Kilbas, H. M. Srivastava, J. J. Trujillo. {\it Theory and Applications of Fractional Differential Equations}. North-Holland Mathematics Studies. 2006.
\bibitem[Nab04]{Naber} M. Naber, Time fractional Schr\"{o}dinger equation, {\it J. Math. Phys.}, 45 (2004), 3339--3352.
\bibitem[Par08]{Par08} I. R. Parissis. A sharp bound for the Stein-Wainger oscillatory integral. {\it Proc. Amer. Math. Soc.} 136:3 (2008), 963--972.
\bibitem[PS92]{PS92} D. H. Phong, E. M. Stein, Oscillatory integrals with polynomial phases, {\it Invent. Math.}, 110:1 (1992) 39--62.
\bibitem[PS94]{PS94} D. H. Phong, E. M. Stein, Models of degenerate Fourier integral operators and Radon transforms, {\it Ann. of Math.} 140:3 (1994), 703--722.
\bibitem[PSS01]{PSS01} D. H. Phong, E. M. Stein and J. Sturm, Multilinear level set operators, oscillatory integral operators, and Newton polyhedra, {\it Math. Ann.} 319 (2001) 573--596.
\bibitem[Pod99]{Pod99} I. Podlubny, {\it Fractional Differential Equations}, Academic Press, New York, 1999.
\bibitem[Rog05]{Rog05} K. M. Rogers, Sharp van der Corput estimates and minimal divided differences, {\it Proc. Amer. Math. Soc.}
133: 12 (2005) 3543--3550.
\bibitem[Ruz12]{Ruz12} M. Ruzhansky, Multidimensional decay in the van der Corput lemma, {\it Studia Math.} 208 (2012) 1--10.
\bibitem[RT20]{RuzTor} M. Ruzhansky, B. T. Torebek. Van der Corput lemmas for Mittag-Leffler functions. ArXiv, 2020, 1-32, arXiv:2002.07492
\bibitem[St93]{St93} E. M. Stein, {\it Harmonic Analysis: Real-Variable Methods, Orthogonality, and Oscillatory Integrals.} Princeton
Mathematical Series. Vol. 43. Princeton Univ. Press, Princeton, 1993.
\bibitem[SW70]{SW70} E. M. Stein, S. Wainger. The estimation of an integral arising in multiplier transformations. {\it Studia Math.} 35 (1970) 101--104.
\bibitem[SZ20]{Su} X. Su, J. Zheng. H\"{o}lder regularity for the time fractional Schr\"{o}dinger equation. {\it Math. Meth. Appl. Sci.} 43:7 (2020), 4847--4870.
\bibitem[vdC21]{vdC21} J. G. van der Corput, Zahlentheoretische Absch\"{a}tzungen, {\it Math. Ann.} 84: 1-2 (1921) 53--79.
\bibitem[Xi17]{Xi17} L. Xiao, Endpoint estimates for one-dimensional oscillatory integral operators. {\it Adv. Math.} 316 (2017) 255--291.
\end{thebibliography}
\end{document}